\documentclass[10pt]{article}
\usepackage{amsfonts}
\usepackage{hyperref}
\usepackage{mathrsfs,enumerate,amssymb,amsmath,color,lineno}
\usepackage{indentfirst}
\usepackage{amsmath}
\usepackage{amssymb}
\usepackage[amsmath,thmmarks]{ntheorem}
\usepackage{graphicx}
\usepackage{tikz}

\newtheorem{definition}{\bf Definition}[section]
\newtheorem{lemma}{\bf Lemma}[section]
\newtheorem{theorem}{\bf Theorem}[section]
\newtheorem{remark}{\bf Remark}[section]
\newtheorem{corollary}{\bf Corollary}[section]
\newtheorem{example}{\bf Example}[section]

\newtheorem{proposition}{\bf Proposition}[section]

\begin{document}
\setcounter{page}{1}

\title{{\textbf{Associativity of discretized triangular norms}}\thanks {Supported by
the National Natural Science Foundation of China (No. 12471440))}}
\author{Di Zhang\footnote{\emph{E-mail address}: 2498586243@qq.com}, Xue-ping Wang\footnote{Corresponding author. xpwang1@hotmail.com; fax: +86-028-84761502},\\
\emph{School of Mathematical Sciences, Sichuan Normal University,}\\
\emph{Chengdu 610066, Sichuan, People's Republic of China}}

\newcommand{\pp}[2]{\frac{\partial #1}{\partial #2}}
\date{}
\maketitle
\begin{quote}
{\bf Abstract}
In this article, we mainly investigate when the binary operation $T^{[n]}:[0, 1]^2\rightarrow [0, 1]$ satisfies the associative law and thus becomes a triangular norm where \begin{equation*}
T^{[n]}(x,y)=\left\{\begin{array}{ll}   f_{n}(T(f_{n}(x),f_{n}(y))),  &\text{if}\,\, (x,y)\in{[0,1)^{2}},   \\
\min\{x,y\},  &\text{otherwise,}  \end{array}\right.
\end{equation*} $f_{n}:[0,1]\rightarrow\left\{0, \frac{1}{n}, \frac{2}{n}, \dots, \frac{n-1}{n}, 1\right\}$ given by $f_{n}(x)=\frac{1}{n}{\lfloor nx \rfloor}$ with $n$ a positive integer and $T:[0, 1]^2\rightarrow [0, 1]$ is a triangular norm. We present some sufficient conditions for $T^{[n]}$ to be triangular norms, and for $T^{[n]}$ to be never associative when $T$ is strictly monotone.

\end{quote}

\section{Introduction}
Triangular norms (t-norms for short) are fundamental aggregation operators on the unit interval $[0,1]$, which play an important role in fuzzy logic, probabilistic metric spaces and many other areas \cite{KM1942,BS1960,BS1961,EP2000}. In many applications, it is desirable to replace the continuous truth-value set $[0,1]$ by a finite chain or poset \cite{GM1993,BD1998b,BD1999}. This leads naturally to the study of discrete t-norms and finite-valued aggregation operators \cite{EP2000,GM2005,AM2007,JC2011}. For a positive integer $n$, consider the finite chain
$\Gamma_{(n)} = \left\{0, \frac{1}{n}, \frac{2}{n}, \dots, \frac{n-1}{n}, 1\right\}$.
A natural approach to constructing a finite-valued version of a given t-norm $T:[0, 1]^2\rightarrow [0, 1]$ is to discretize its arguments and its output by means of the floor function
$f_n(x)=\frac1n\lfloor nx\rfloor$.
Following \cite{EP2000}, this leads to the binary operation $T^{[n]}:[0, 1]^2\rightarrow [0, 1]$ defined by
\begin{equation*}
T^{[n]}(x,y)=\left\{\begin{array}{ll}   f_{n}(T(f_{n}(x),f_{n}(y))),  &\text{if}\,\, (x,y)\in{[0,1)^{2}},   \\
\min\{x,y\},  &\text{otherwise,}  \end{array}\right.
\end{equation*}
 which automatically inherits commutativity, monotonicity, and the neutral element from the t-norm $T$, whereas the associativity is not preserved in general. Thus, the central question is to determine when $T^{[n]}$ is associative and hence itself a t-norm. The difficulty comes from the interaction between the algebraic structure of $T$ and the discretization process. Observe that the map $f_n$ identifies a whole interval of points with a single element of $\Gamma_{(n)}$, so that some information is lost. Although the loss becomes smaller as $n$ increases, it is sufficient to affect the associativity of $T^{[n]}$. It was claimed in \cite[p.100]{EP2000} that if $T$ is strictly monotone, then $T^{[n]}$ is never associative for $n\geq2$. However, we find that this assertion is not always valid by a counterexample. This observation motivates us to explore the conditions governing the associativity of $T^{[n]}$.

The remainder of this article is organized as follows. Section 2 recalls the necessary concepts and preliminary results. Section 3 contains the main results on the associativity of $T^{[n]}$, including sufficient conditions for its preservation, the non-associativity results for $n\geq4$. A conclusion is drawn in Section 4.

\section{Preliminaries}
In this section, we review some necessary concepts and results. For more comprehensive presentation, see book \cite{EP2000}.
\begin{definition}[\cite{EP2000,CA2006}]\label{defi 2.1}
\emph{A \emph{triangular norm} (\emph{t-norm} for short) is a binary operator $T:[0, 1]^2\rightarrow [0, 1]$ such that for all $x, y, z\in[0, 1]$ the following conditions are satisfied:}

$(T1)$  $T(x,y)=T(y,x)$,

$(T2)$  $T(T(x,y),z)=T(x,T(y,z))$,

\emph{$(T3)$  $T(x,y)\leq T(x,z)$ whenever $y\leq z$,}

$(T4)$  $T(x,1)=T(1,x)=x$.
\end{definition}

The associativity (T2) allows us to extend each t-norm $T$ in a unique way to an $n$-ary operation in usual way by induction. For each $x\in{[0,1]}$, we define $$x_{T}^{(0)}=1, x_{T}^{(1)}=x, x_{T}^{(n+1)}=T(x_{T}^{(n)},x) \mbox{ where }n\in \mathbb{N}=\{1, 2, \cdots, n,\cdots\}.$$

\begin{definition}[\cite{EP2000,CA2006}]\label{defi 2.2}
\emph{A \emph{triangular conorm} (\emph{t-conorm} for short) is a binary operator $S:[0, 1]^2\rightarrow [0, 1]$ such that for all $x, y, z\in[0, 1]$ the following conditions are satisfied:}

$(S1)$ $S(x,y)=S(y,x)$,

$(S2)$  $S(S(x,y),z)=S(x,S(y,z))$,

\emph{$(S3)$  $S(x,y)\leq S(x,z)$ whenever $y\leq z$,}

$(S4)$ $S(x,0)=S(0,x)=x$.
\end{definition}

Let $f_{n}(x)=\frac{1}{n}{\lfloor nx \rfloor}$, where $n \in {\mathbb{N}}$ and ${\lfloor nx \rfloor}$ is the greatest integer which is smaller than or equal to $nx$. After a simple computation, we can obtain
\begin{equation}\label{eq 03}
f_{n}(x)=\frac{1}{n}{\lfloor nx \rfloor}=\left\{\begin{array}{rcl}   0,  &\text{}x\in{[0,\frac{1}{n})},   \\\\
\frac{1}{n}, &\text{}x\in{[\frac{1}{n},\frac{2}{n})},\\\\
 \frac{2}{n}, &\text{}x\in{[\frac{2}{n},\frac{3}{n})},\\
 \vdots &\text{}\vdots\\
 \frac{n-1}{n}, &\text{}x\in{[\frac{n-1}{n},1)},\\
 1, &\text{}x=1.\\
  \end{array}\right.
 \end{equation}
Obviously, $f_{n}:[0,1]\rightarrow{\Gamma_{(n)}}$ is a non-decreasing and right-continuous function, where $$\Gamma_{(n)} = \left\{0, \frac{1}{n}, \frac{2}{n}, \dots, \frac{n-1}{n}, 1\right\}.$$

\begin{lemma}[\cite{EP2000}]\label{lem 2.1}
For each $n \in \mathbb{N}$ define the function $f_{n}:[0,1]\rightarrow{\Gamma_{(n)}}$ by $f_{n}(x)=\frac{1}{n}{\lfloor nx \rfloor}$. Let $T:[0,1]^{2}\rightarrow{[0,1]}$ be a t-norm. Then the binary operation $T^{[n]}:[0,1]^{2}\rightarrow{[0,1]}$ given by
\begin{equation}\label{eq 04}
T^{[n]}(x,y)=\left\{\begin{array}{ll}   f_{n}(T(f_{n}(x),f_{n}(y))),  &\text{if}\,\, (x,y)\in{[0,1)^{2}},   \\
\min\{x,y\},  &\text{otherwise.}  \end{array}\right.
\end{equation}
satisfies the axioms (T1), (T3) and (T4), but not necessarily (T2).
\end{lemma}

From \cite[p.100]{EP2000}, $T^{[n]}$ is associative (and, subsequently, a t-norm) if and only if $(\Gamma_{(n)},T^{[n]}|_{\Gamma_{(n)}^{2}},\leq)$ is a commutative, fully ordered semigroup with neutral element $1$, i.e., if $T^{[n]}|_{\Gamma_{(n)}^{2}}$ is a discrete t-norm. For instance, if $T=T_{M}$, or $T=T_{L}$, or $T=T_{D}$, then $T^{[n]}$ must be associative, furthermore, $T^{[n]}$ is a t-norm for $n\in{\mathbb{N}}$.

In \cite[p.100]{EP2000}, the authors indicate that if a t-norm $T$ is strictly monotone then $T^{[n]}$ is never associative for $n\geq2$. However, this result may not necessarily be true. The subsequent example can explain this viewpoint clearly.
\begin{example}\label{exp 2.1}
\emph{From Lemma \ref{lem 2.1}, the binary operation $T^{[n]}:[0,1]^2\rightarrow{[0,1]}$ given by Eq.\eqref{eq 04} satisfies (T1), (T3) and (T4), but it is not necessarily associative. On the other hand, the following facts hold:}
\renewcommand{\labelenumi}{\emph{(\roman{enumi})}}
\begin{enumerate}
\item \emph{If $n=2$ and $T=T_{P}$, then by Eq.\eqref{eq 03}, for all $x,y,z\in{[0,1)}$ we have $f_{2}(x), f_{2}(y), f_{2}(z)\in{\{0,\frac{1}{2}\}}$ and $T^{[2]}(x,y)=f_{2}(T_{P}(f_{2}(x),f_{2}(y)))=f_{2}(f_{2}(x)\cdot{f_{2}(y)})\leq{f_{2}(\frac{1}{4})}=0$. Thus
$$T^{[2]}(T^{[2]}(x,y),z)=f_{2}(T_{P}(f_{2}(T^{[2]}(x,y)),f_{2}(z))=f_{2}(0)=0.$$ In complete analogy, we have $T^{[2]}(x,T^{[2]}(y,z))=0$. Hence $T^{[2]}(T^{[2]}(x,y),z)=T^{[2]}(x,T^{[2]}(y,z))$ for all $x,y,z\in{[0,1)}$. If at least one of $\{x,y,z\}$ equals to $1$, then obviously, we have $T^{[2]}(T^{[2]}(x,y),z)=T^{[2]}(x,T^{[2]}(y,z))$. Therefore, $T^{[2]}(T^{[2]}(x,y),z)=T^{[2]}(x,T^{[2]}(y,z))$ for all $x,y,z\in{[0,1]}$.}
\item \emph{If $n=3$ and $T=T_{P}$, then by Eq.\eqref{eq 03}, for arbitrary $x,y,z\in{[0,1)}$ we get $f_{3}(x), f_{3}(y), f_{3}(z)\in{\{0, \frac{1}{3}, \frac{2}{3}\}} $. Thus $$T^{[3]}(x,y)=f_{3}(T_{P}(f_{3}(x),f_{3}(y)))=f_{3}(f_{3}(x)\cdot{f_{3}(y)})\in{\{0,\frac{1}{3}\}},$$
which together with $T^{[3]}(x,y)\cdot{f_{3}(z)}\leq{\frac{2}{9}}<{\frac{1}{3}}$ implies
\begin{eqnarray*}
T^{[3]}(T^{[3]}(x,y),z)&=&f_{3}(T_{P}(f_{3}(T^{[3]}(x,y)),f_{3}(z)))\\
&=&f_{3}(T^{[3]}(x,y)\cdot{f_{3}(z)})\\
&=&0.
\end{eqnarray*}
Analogously, we can obtain $T^{[3]}(x,T^{[3]}(y,z))=0$, i.e., $T^{[3]}(T^{[3]}(x,y),z)=0=T^{[3]}(x,T^{[3]}(y,z))$ for all $x,y,z\in{[0,1)}$.
If at least one of $\{x,y,z\}$ equals to $1$, then obviously, $T^{[3]}(T^{[3]}(x,y),z)=T^{[3]}(x,T^{[3]}(y,z))$. Therefore, $T^{[3]}(T^{[3]}(x,y),z)=T^{[3]}(x,T^{[3]}(y,z))$ for all $x,y,z\in{[0,1]}$.}
\item \emph{If $n\geq{4}$ and $T=T_{P}$, then $T^{[n]}$ is never associative.}

 \emph{Indeed, by Eq.\eqref{eq 03} we have that $f_{n}(x),f_{n}(y),f_{n}(z)\in{\{0,\frac{1}{n},\frac{2}{n},\dots, \frac{n-1}{n}\}}$ for any $x,y,z\in{[0,1)}$.
Now we choose $x,y\in{[\frac{n-1}{n},1)}, z\in{[\frac{2}{n},\frac{3}{n})}$, then again by Eq.\eqref{eq 03}, $f_{n}(x)=f_{n}(y)=\frac{n-1}{n}$ and $f_{n}(z)=\frac{2}{n}$, which deduce $\frac{n-2}{n}<{\frac{(n-1)^{2}}{n^{2}}}<{\frac{n-1}{n}}$ since $n\geq{4}$. Thus
$$T^{[n]}(x,y)=f_{n}(T_{P}(f_{n}(x),f_{n}(y)))=f_{n}(T_{P}(\frac{n-1}{n},\frac{n-1}{n}))=f_{n}(\frac{(n-1)^{2}}{n^{2}})=\frac{n-2}{n}.$$
Note that $\frac{1}{n}\leq{\frac{2(n-2)}{n^{2}}}<{\frac{2n-2}{n^{2}}}<{\frac{2}{n}}$ when $n\geq{4}$. Then we immediately have
\begin{equation}\label{001eq}
T^{[n]}(T^{[n]}(x,y),z)=f_{n}(T_{P}(\frac{n-2}{n},\frac{2}{n}))=f_{n}(\frac{2(n-2)}{n^{2}})=\frac{1}{n}.
\end{equation}
Similarly, we can obtain
$$T^{[n]}(y,z)=f_{n}(T_{P}(f_{n}(y),f_{n}(z)))=f_{n}(T_{P}(\frac{n-1}{n},\frac{2}{n}))=f_{n}(\frac{2n-2}{n^{2}})=\frac{1}{n}$$
and
\begin{equation}\label{002eq}
T^{[n]}(x,T^{[n]}(y,z))=T^{[n]}(x,\frac{1}{n})=f_{n}(T_{P}(\frac{n-1}{n},\frac{1}{n}))=0.\end{equation}
Eqs. \eqref{001eq} and \eqref{002eq} mean that $T^{[n]}(T^{[n]}(x,y),z)\neq{T^{[n]}(x,T^{[n]}(y,z))}.$ Therefore, $T^{[n]}$ is never associative, i.e., $T^{[n]}$ is not a t-norm  when $T=T_{P}$ and $n\geq{4}$.}
\end{enumerate}
\end{example}

\begin{remark}\label{rem 2.3}
\emph{According to Example \ref{exp 2.1}, two questions naturally arise:}
\renewcommand{\labelenumi}{\emph{(\roman{enumi})}}
\begin{enumerate}
\item \emph{What conditions guarantee that $T^{[n]}$ given by Eq.\eqref{eq 04} is a t-norm?}
\item \emph{What conditions are required for $T^{[n]}$ given by Eq.\eqref{eq 04} is never associative whenever $T$ is strictly monotone?}
\end{enumerate}
\end{remark}

\section{Main results}

In this section, we present some sufficient conditions for $T^{[n]}$ to be t-norms, and for $T^{[n]}$ to be never associative when $T$ is strictly monotone.
\begin{theorem}\label{the 3.1}
Let $T:[0,1]^2\rightarrow[0,1]$ be a t-norm. If $T(x,y)<T_M(x,y)$ for all $x,y\in(0,1)$ and $T(\frac{i}{n},\frac{j}{n})\in{[0,\frac{2}{n})}$ for all $i,j\in{\{0,1,2,\dots, n-1\}}$, then $T^{[n]}$ given by Eq.\eqref{eq 04} is a t-norm.
\end{theorem}
\begin{proof}
As a consequence of Lemma \ref{lem 2.1}, $T^{[n]}$ satisfies (T1), (T3) and (T4), so it is enough to prove that $T^{[n]}$ satisfies (T2).

We first suppose that $x,y,z\in{[0,1)}$. Then from Eq.\eqref{eq 03}, we can let $f_{n}(x)=\frac{i}{n}, f_{n}(y)=\frac{j}{n}$ and $f_{n}(z)=\frac{k}{n}$ with $i,j,k\in{\{0,1,2,\dots, n-1\}}$. There are three cases as follows.

Case 1. If $T(\frac{i}{n},\frac{j}{n}), T(\frac{j}{n},\frac{k}{n})\in{[0,\frac{1}{n})}$, then from Eq.\eqref{eq 03} it is easy to see
$$T^{[n]}(x,y)=f_{n}(T(f_{n}(x),f_{n}(y)))=f_{n}(T(\frac{i}{n},\frac{j}{n}))=0.$$
Similarly, we have $ T^{[n]}(y,z)=0.$  Thus
$$T^{[n]}(T^{[n]}(x,y),z)=f_{n}(T(0,f_{n}(z)))=f_{n}(T(0,\frac{k}{n}))=f_{n}(0)=0$$
and
$$T^{[n]}(x,T^{[n]}(y,z))=f_{n}(T(f_{n}(x),0))=f_{n}(T(\frac{i}{n},0))=f_{n}(0)=0.$$
Consequently, $T^{[n]}(T^{[n]}(x,y),z)=T^{[n]}(x,T^{[n]}(y,z))$.

Case 2. If $T(\frac{i}{n},\frac{j}{n}), T(\frac{j}{n},\frac{k}{n})\in{[\frac{1}{n},\frac{2}{n})}$, then
$$T^{[n]}(x,y)=f_{n}(T(f_{n}(x),f_{n}(y)))=f_{n}(T(\frac{i}{n},\frac{j}{n}))=\frac{1}{n},$$
and analogously, we have $T^{[n]}(y,z)=\frac{1}{n}$. Since $T(x,y)<T_M(x,y)$ for all $x,y\in(0,1)$, we have that $$T(\frac{1}{n},\frac{k}{n})<T_{M}(\frac{1}{n},\frac{k}{n})\leq\frac{1}{n}.$$ Again from Eq.\eqref{eq 03},
$$T^{[n]}(T^{[n]}(x,y),z)=f_{n}(T(f_{n}(T^{[n]}(x,y)),f_{n}(z)))=f_{n}(T(f_{n}(\frac{1}{n}),\frac{k}{n}))=f_{n}(T(\frac{1}{n},\frac{k}{n}))=0.$$
In complete analogy, we get $T^{[n]}(x,T^{[n]}(y,z))=0.$ Therefore, $T^{[n]}(T^{[n]}(x,y),z)=T^{[n]}(x,T^{[n]}(y,z))$.

Case 3. If either $T(\frac{i}{n},\frac{j}{n})\in{[0,\frac{1}{n})}$ and $T(\frac{j}{n},\frac{k}{n})\in{[\frac{1}{n},\frac{2}{n})}$ or $T(\frac{j}{n},\frac{k}{n})\in{[0,\frac{1}{n})}$ and $T(\frac{i}{n},\frac{j}{n})\in{[\frac{1}{n},\frac{2}{n})}$, say $T(\frac{i}{n},\frac{j}{n})\in{[0,\frac{1}{n})}$ and $T(\frac{j}{n},\frac{k}{n})\in{[\frac{1}{n},\frac{2}{n})}$, then as immediate consequences of Cases 1 and 2 we get $T^{[n]}(x,y)=0$ and $T^{[n]}(y,z)=\frac{1}{n}.$ Consequently,
$$T^{[n]}(T^{[n]}(x,y),z)=T^{[n]}(0,z)=f_{n}(T(f_{n}(0),f_{n}(z)))=f_{n}(T(0,\frac{k}{n}))=f_{n}(0)=0.$$
Analogously, we have
$$T^{[n]}(x,T^{[n]}(y,z))=f_{n}(T(f_{n}(x),f_{n}(\frac{1}{n})))=f_{n}(T(\frac{i}{n},\frac{1}{n}))=0.$$
Therefore, $T^{[n]}(T^{[n]}(x,y),z)=T^{[n]}(x,T^{[n]}(y,z)).$

Cases 1, 2 and 3 mean that $T^{[n]}(T^{[n]}(x,y),z)=T^{[n]}(x,T^{[n]}(y,z))$ for any $x,y,z\in{[0,1)}$.

Now, if at least one of $x,y,z$ equals $1$, then $T^{[n]}(T^{[n]}(x,y),z)=T^{[n]}(x,T^{[n]}(y,z))$ holds trivially.

Therefore, $T^{[n]}$ is a t-norm.
\end{proof}

\begin{corollary}\label{cor 3.1}
(i) If $T:[0,1]^{2}\rightarrow{[0,1]}$ is an Archimedean t-norm and $T(\frac{i}{n},\frac{j}{n})\in{[0,\frac{2}{n})}$ for all $i,j\in{\{0,1,2,\dots, n-1\}}$, then $T^{[n]}$ given by Eq.\eqref{eq 04} is a t-norm.

(ii) If a t-norm $T:[0,1]^2\rightarrow[0,1]$ is strictly monotone and $T(\frac{i}{n},\frac{j}{n})\in{[0,\frac{2}{n})}$ for all $i,j\in{\{0,1,2,\dots, n-1\}}$, then $T^{[n]}$ given by Eq.\eqref{eq 04} is a t-norm.

(iii) If $T:[0,1]^2\rightarrow[0,1]$ is a t-norm with an additive generator $t$ satisfying $t(\frac{i}{n})+t(\frac{j}{n})\in{\emph{Ran}(t\mid_{[0,\frac{2}{n})}})$ for all $i,j\in{\{0,1,2,\dots, n-1\}}$, then $T^{[n]}$ given by Eq.\eqref{eq 04} is a t-norm.
\end{corollary}
\begin{proof}
(i) It is clear to see $T(x,y)\leq\min\{x,y\}$ for all $x,y\in{(0,1)}$. Based on Theorem 2.12 (iii) of \cite{EP2000}, without loss of generality, let $x<y$. We claim that $T(x,y)<T_{M}(x,y)$.
Assume that $T(x,y)=T_{M}(x,y)$. According to (T2), it is immediately to get $$x=T(x,y)=T(T(x,y),y)=T(x,T(y,y))=T(x,y_{T}^{(2)}).$$
By induction, we have $x=T(x,y_{T}^{(n)})$ for each $n\in\mathbb{N}$. Since $T$ is Archimedean, there exists an $n_{0}\in\mathbb{N}$ such that $y_{T}^{(n_{0})}<x$, which together with Theorem 2.12 (iii) of \cite{EP2000} leads to a contradiction $$x=T(x,y_{T}^{(n_{0})})\leq{T(x,x)}<x.$$
Therefore, by Theorem \ref{the 3.1}, $T^{[n]}$ given by Eq.\eqref{eq 04} is a t-norm.

(ii) Since $T$ is strictly monotone, $T(x,y)<T_M(x,y)$ for all $x,y\in(0,1)$. By Theorem \ref{the 3.1}, $T^{[n]}$ is a t-norm.

(iii) By Proposition 3.29 of \cite{EP2000}, $T$ is Archimedean. Since $t(\frac{i}{n})+t(\frac{j}{n})\in{\mbox{Ran}(t\mid_{[0,\frac{2}{n})})}$ for all $i,j\in{\{0,1,2,\dots, n-1\}}$, there exists a $w\in{[0,\frac{2}{n})}$ such that $t(\frac{i}{n})+t(\frac{j}{n})=t(w)$.  From (3.16) of \cite{EP2000}, a simple computation leads to $$T(\frac{i}{n},\frac{j}{n})=t^{(-1)}(t(\frac{i}{n})+t(\frac{j}{n}))=w\in{[0,\frac{2}{n})} \ \mbox{for} \ \mbox{all} \ i,j\in{\{0,1,2,\dots, n-1\}}.$$
Therefore, it follows from (i) that $T^{[n]}$ is a t-norm.
\end{proof}

\begin{corollary}\label{cor 3.2}
Let $T:[0,1]^{2}\rightarrow{[0,1]}$ be a t-norm. If one of the following condition holds,\renewcommand{\labelenumi}{(\roman{enumi})}
\begin{enumerate}
\item $T(x,y)<T_M(x,y)$ for all $x,y\in(0,1)$,
\item $T$ is Archimedean,
\item $T$ is strictly monotone,
\end{enumerate}then both $T^{[2]}$ and $T^{[3]}$ are t-norms.
\end{corollary}
\begin{proof}
Whenever $T$ satisfies any one of the three conditions (i), (ii) and (iii), we can obtain $T(\frac{i}{2},\frac{j}{2})\in{[0,\frac{1}{2})}, T(\frac{j}{2},\frac{k}{2})\in{[0,\frac{1}{2})}$ for all $i,j,k\in{\{0,1\}}$ and $T(\frac{i}{3},\frac{j}{3})\in{[0,\frac{2}{3})}, T(\frac{j}{3},\frac{k}{3})\in{[0,\frac{2}{3})}$ for all $i,j,k\in{\{0,1,2\}}$.

Therefore, by Theorem \ref{the 3.1} and Corollary \ref{cor 3.1}, both $T^{[2]}$ and $T^{[3]}$ are t-norms.
\end{proof}
\begin{remark}\label{rem 3.1} \emph{Let $T:[0,1]^2\rightarrow[0,1]$ be a t-norm.
\renewcommand{\labelenumi}{(\roman{enumi})}
\begin{enumerate}
\item  Example \ref{exp 2.1} shows that for $n>3$, $T^{[n]}$ given by Eq.\eqref{eq 04} is not necessarily associative even if $T$ satisfies one of the conditions (i), (ii) and (iii) of Corollary \ref{cor 3.2}.
\item  From Remark 7.27 (i) of \cite{EP2000}, we know that $T^{[n]}$ given by Eq.\eqref{eq 04} is a t-norm when $T|_{\Gamma_{(n)}^{2}}$ is a discrete t-norm. However, from Corollary \ref{cor 3.1}, $T^{[n]}$ is a t-norm even if $T|_{\Gamma_{(n)}^{2}}$ is not necessarily a discrete t-norm. For example, define $T:[0,1]^2\rightarrow[0,1]$ by $$T(x,y)=\left\{\begin{array}{ll}  \frac{xy}{n},  &\text{}(x,y)\in{[0,1)^{2}},\\
\min\{x,y\},  &\text{otherwise.}\end{array}\right.$$
Obviously, $T$ is a strictly monotone, non-continuous and Archimedean t-norm. From Eq.\eqref{eq 03}, for any $x,y,z\in{[0,1)}$, let $f_{n}(x)=\frac{i}{n}, f_{n}(y)=\frac{j}{n}$ and $f_{n}(z)=\frac{k}{n}$ with $i,j,k\in{\{0,1,2,\dots, n-1\}}$. After a simple computation, we get $$T(\frac{i}{n},\frac{j}{n})=\frac{\frac{i}{n}\cdot{\frac{j}{n}}}{n}<\frac{1}{n}\in[0,\frac{2}{n}) \ \mbox{for}\ \mbox{all}\ i, j\in{\{0,1,2,\dots, n-1\}}.$$ According to Corollary \ref{cor 3.1}, $T^{[n]}$ is a t-norm. Moreover, one can check that $\Gamma_{(n)}$ is not closed under $T$, i.e., $T|_{\Gamma_{(n)}^{2}}$ is not a discrete t-norm.
\end{enumerate}}
\end{remark}

Next, we give a sufficient condition for $T^{[n]}$ to be a t-norm whenever $T:[0,1]^2\rightarrow[0,1]$ is a continuous t-norm.
\begin{theorem}\label{the 3.2}
 Let $T:[0,1]^2\rightarrow[0,1]$ be uniquely representable as an ordinal sum of continuous Archimedean t-norms $(T_\alpha)_{\alpha\in A}$ with $A\neq\emptyset$. If for all $(\frac{i}{n},\frac{j}{n})\in[a_{\alpha},e_{\alpha}]^2$ and $\alpha\in A$, there exists $\frac{p}{n}\in[a_{\alpha},e_{\alpha}]$ such that $T_{\alpha}(\varphi_{\alpha}(\frac{i}{n}),\varphi_{\alpha}(\frac{j}{n}))=\varphi_{\alpha}(\frac{p}{n})$ where $\varphi_{\alpha}(x)=\frac{x-a_{\alpha}}{e_{\alpha}-a_{\alpha}}$ and $i, j, p\in{\{0,1,2,\dots, n-1\}}$, then $T^{[n]}$ given by Eq.\eqref{eq 04} is a t-norm.
\end{theorem}
\begin{proof}
 By Lemma \ref{lem 2.1}, only the associativity of $T^{[n]}$ is not trivial. It is easy to check that $T^{[n]}(T^{[n]}(x,y),z)=T^{[n]}(x,T^{[n]}(y,z))$ holds whenever $1\in{\{x,y,z\}}$.

 Now, suppose that $x,y,z\in{[0,1)}$. Then from Eq.\eqref{eq 03}, we can let $f_{n}(x)=\frac{i}{n}, f_{n}(y)=\frac{j}{n}, f_{n}(z)=\frac{k}{n}$ with $i, j, k\in{\{0,1,2,\dots, n-1\}}$. The rest of the proof can be divided into three cases as follows.

Case 1: If $\frac{i}{n}, \frac{j}{n}$ and $\frac{k}{n}$ belong to three distinct intervals, then without loss of generality, let $\frac{i}{n}\in{[a_{\alpha},e_{\alpha}]}$, $\frac{j}{n}\in{[a_{\beta},e_{\beta}]}$, $\frac{k}{n}\in{[a_{\gamma},e_{\gamma}]}$. Thus we have
$$T^{[n]}(x,y)=f_{n}(T(f_{n}(x),f_{n}(y)))=f_{n}(T(\frac{i}{n},\frac{j}{n}))=f_{n}(T_{M}(\frac{i}{n},\frac{j}{n}))=T_{M}(\frac{i}{n},\frac{j}{n}).$$
Analogously, we have $T^{[n]}(y,z)=f_{n}(T_{M}(\frac{j}{n},\frac{k}{n}))=T_{M}(\frac{j}{n},\frac{k}{n}).$
Consequently,
 \begin{eqnarray*}
T^{[n]}(x,T^{[n]}(y,z))&=&f_{n}(T(f_{n}(x),f_{n}(T^{[n]}(y,z))))\\
&=&f_{n}(T(\frac{i}{n},T^{[n]}(y,z)))\\
&=&f_{n}(T_{M}(\frac{i}{n},T_{M}(\frac{j}{n},\frac{k}{n})))\\
&=&f_{n}(T_{M}(T_{M}(\frac{i}{n},\frac{j}{n}),\frac{k}{n}))\mbox{ by (T2)}\\
&=&f_{n}(T(T^{[n]}(x,y),f_{n}(z)))\\
&=&T^{[n]}(T^{[n]}(x,y),z).
\end{eqnarray*}

Case 2: If $\frac{i}{n}, \frac{j}{n}$ and $\frac{k}{n}$ belong to two distinct intervals, say
$\frac{i}{n}, \frac{j}{n}\in[a_{\alpha},e_{\alpha}]$ and $\frac{k}{n}\in{[a_{\beta},e_{\beta}]}$, then by assumption, there exists $\frac{p}{n}\in{[a_{\alpha},e_{\alpha}]}$ such that
$T_{\alpha}(\varphi_{\alpha}(\frac{i}{n}),\varphi_{\alpha}(\frac{j}{n}))=\varphi_{\alpha}(\frac{p}{n})$. Subsequently, by the definition of $\varphi_{\alpha}$ we can obtain
 \begin{eqnarray*}
T^{[n]}(x,y)&=&f_{n}(T(f_{n}(x),f_{n}(y)))=f_{n}(T(\frac{i}{n},\frac{j}{n}))\\
&=&f_{n}(a_{\alpha}+(e_{\alpha}-a_{\alpha})T_{\alpha}(\varphi_{\alpha}(\frac{i}{n}),\varphi_{\alpha}(\frac{j}{n}))\\
&=&f_{n}(a_{\alpha}+(e_{\alpha}-a_{\alpha})\varphi_{\alpha}(\frac{p}{n}))\\
&=&f_{n}(\varphi_{\alpha}^{-1}(\varphi_{\alpha}(\frac{p}{n})))\\
&=&f_{n}(\frac{p}{n})\\
&=&\frac{p}{n}
\end{eqnarray*}
and $T^{[n]}(y,z)=f_{n}(T(f_{n}(y),f_{n}(z)))=f_{n}(T_{M}(\frac{j}{n},\frac{k}{n}))=T_{M}(\frac{j}{n},\frac{k}{n})$. Thus \begin{equation}\label{eq07}T^{[n]}(T^{[n]}(x,y),z)=f_{n}(T(f_{n}(\frac{p}{n}),f_{n}(z)))=f_{n}(T(\frac{p}{n},\frac{k}{n}))=f_{n}(T_{M}(\frac{p}{n},\frac{k}{n}))\end{equation}
and \begin{equation}\label{eq08}T^{[n]}(x,T^{[n]}(y,z))=f_{n}(T(\frac{i}{n},f_{n}(T_{M}(\frac{j}{n},\frac{k}{n})))).\end{equation}
On the one hand, if $e_{\alpha}\leq{a_{\beta}}$, then we have $T_{M}(\frac{p}{n},\frac{k}{n})=\frac{p}{n}$ and $T_{M}(\frac{j}{n},\frac{k}{n})=\frac{j}{n}$, which together with Eqs.\eqref{eq07} and \eqref{eq08} imply that
$$T^{[n]}(T^{[n]}(x,y),z)=f_{n}(\frac{p}{n})=\frac{p}{n}$$ and $$T^{[n]}(x,T^{[n]}(y,z))=f_{n}(T(\frac{i}{n},\frac{j}{n}))=f_{n}(\frac{p}{n})=\frac{p}{n}.$$
On the other hand, if $e_{\beta}\leq{a_{\alpha}}$, then we can similarly obtain $T^{[n]}(T^{[n]}(x,y),z)=\frac{k}{n}=T^{[n]}(x,T^{[n]}(y,z))$.

Therefore, $T^{[n]}(T^{[n]}(x,y),z)=T^{[n]}(x,T^{[n]}(y,z))$.

Case 3: If $\frac{i}{n}$, $\frac{j}{n}$ and $\frac{k}{n}$ belong to the same interval, i.e., $\frac{i}{n}, \frac{j}{n}, \frac{k}{n}\in{[a_{\alpha},e_{\alpha}]}$ for some $\alpha\in{A}$, then by assumption, there exist $\frac{p}{n},\frac{q}{n}\in{[a_{\alpha},e_{\alpha}]}$ such that
$T_{\alpha}(\varphi_{\alpha}(\frac{i}{n}),\varphi_{\alpha}(\frac{j}{n}))=\varphi_{\alpha}(\frac{p}{n})$ and $T_{\alpha}(\varphi_{\alpha}(\frac{j}{n}),\varphi_{\alpha}(\frac{k}{n}))=\varphi_{\alpha}(\frac{q}{n})$.
Thus
$$T(\frac{i}{n},\frac{j}{n})=a_{\alpha}+(e_{\alpha}-a_{\alpha})T_{\alpha}(\varphi_{\alpha}(\frac{i}{n}),\varphi_{\alpha}(\frac{j}{n}))=\varphi_{\alpha}^{-1}(\varphi_{\alpha}(\frac{p}{n}))=\frac{p}{n},$$
and similarly, $T(\frac{j}{n},\frac{k}{n})=\frac{q}{n}$. Subsequently,
$$T^{[n]}(x,y)=f_{n}(T(f_{n}(x),f_n(y)))=f_{n}(T(\frac{i}{n},\frac{j}{n}))=T(\frac{i}{n},\frac{j}{n})=\frac{p}{n}\in{\Gamma_{(n)}}$$
and $$ T^{[n]}(y,z)=f_{n}(T(\frac{j}{n},\frac{k}{n}))=T(\frac{j}{n},\frac{k}{n})=\frac{q}{n}\in{\Gamma_{(n)}}.$$
Consequently, we obtain
 \begin{eqnarray*}
 T^{[n]}(T^{[n]}(x,y),z)&=&f_{n}(T(f_{n}(T^{[n]}(x,y)),f_{n}(z)))\\
 &=&f_{n}(T(T(\frac{i}{n},\frac{j}{n}),\frac{k}{n}))\\
 &=&f_{n}(T(\frac{i}{n},T(\frac{j}{n},\frac{k}{n})))\quad\mbox{ by (T2)}\\
 &=&f_{n}(T(f_{n}(x),f_{n}(T(\frac{j}{n},\frac{k}{n}))))\\
 &=&T^{[n]}(x,T^{[n]}(y,z)).
 \end{eqnarray*}

 Cases 1 , 2 and 3 imply $T^{[n]}(T^{[n]}(x,y),z)=T^{[n]}(x,T^{[n]}(y,z))$ for any $x,y,z\in [0,1)$.

 Therefore, $T^{[n]}$ is a t-norm.
\end{proof}
\begin{example}\label{exp 3.1}
\emph{Let $T=(\langle 0,\frac{1}{3}, T_{\alpha_{1}}\rangle, \langle \frac{1}{3},1,T_{\alpha_{2}} \rangle )$ be an ordinal sum of $T_{\alpha_{1}}$ and $T_{\alpha_{2}}$, where $T_{\alpha_{i}}$ is a nilpotent t-norm with an additive generator $t_{\alpha_{i}}(x)=e_{\alpha_{i}}-a_{\alpha_{i}}-(e_{\alpha_{i}}-a_{\alpha_{i}})x$, $i=1,2$. Then $T^{[n]}$ given by Eq.\eqref{eq 04} is a t-norm.}

\emph{In fact, for all $\frac{i}{n},\frac{j}{n}\in[0,\frac{1}{3}]$,
\begin{eqnarray*}
T(\frac{i}{n},\frac{j}{n})&=&\frac{1}{3}T_{\alpha_{1}}(\frac{\frac{i}{n}-0}{\frac{1}{3}-0},\frac{\frac{j}{n}-0}{\frac{1}{3}-0})\\&=&
\frac{1}{3}t_{\alpha_{1}}^{(-1)}(t_{\alpha_{1}}(\frac{\frac{i}{n}}{\frac{1}{3}})+t_{\alpha_{1}}(\frac{\frac{j}{n}}{\frac{1}{3}})) \quad\quad\mbox{ from (3.16) of \cite{EP2000}}\\
&=&\frac{1}{3}t_{\alpha_{1}}^{(-1)}(\frac{1}{3}-\frac{i}{n}+\frac{1}{3}-\frac{j}{n})\\
&=&\frac{1}{3}t_{\alpha_{1}}^{(-1)}(\frac{2}{3}-\frac{i}{n}-\frac{j}{n})\\
&=&\frac{1}{3}\max\{0,3(\frac{1}{3}-(\frac{2}{3}-\frac{i}{n}-\frac{j}{n}))\}\\
&=&\max\{0,\frac{i}{n}+\frac{j}{n}-\frac{1}{3}\}.
\end{eqnarray*}
Thus $T(\frac{i}{n},\frac{j}{n})\in[0,\frac{1}{3}]$.
For every $\alpha_{i}$, consider the linear bijection $\varphi_{\alpha_i}: [a_{\alpha_i}, e_{\alpha_i}]\to [0,1]$ given by $\varphi_{\alpha_i}(x)=\frac{x-a_{\alpha_i}}{e_{\alpha_i}-a_{\alpha_i}}$. Take $\frac{p}{n}=\max\{0,\frac{i}{n}+\frac{j}{n}-\frac{1}{3}\}$. Then we get $$T_{\alpha_1}(\varphi_{\alpha_1}(\frac{i}{n}),\varphi_{\alpha_1}(\frac{j}{n}))=\varphi_{\alpha_1}(\frac{p}{n}) \ \mbox{for}\ \mbox{all}\ \frac{i}{n},\frac{j}{n}\in[0,\frac{1}{3}].$$
For all $\frac{i}{n},\frac{j}{n}\in[\frac{1}{3},1]$, we obtain
\begin{eqnarray*}
T(\frac{i}{n},\frac{j}{n})&=&\frac{1}{3}+\frac{2}{3}T_{\alpha_{2}}(\frac{\frac{i}{n}-\frac{1}{3}}{1-\frac{1}{3}},\frac{\frac{j}{n}-\frac{1}{3}}{1-\frac{1}{3}})\\
&=&\frac{1}{3}+\frac{2}{3}t_{\alpha_{2}}^{(-1)}(t_{\alpha_{2}}(\frac{\frac{i}{n}-\frac{1}{3}}{\frac{2}{3}})+t_{\alpha_{2}}(\frac{\frac{j}{n}-\frac{1}{3}}{\frac{2}{3}})) \quad\quad\mbox{ from (3.16) of \cite{EP2000}}\\
&=&\frac{1}{3}+\frac{2}{3}t_{\alpha_{2}}^{(-1)}(1-\frac{i}{n}+1-\frac{j}{n})\\
&=&\frac{1}{3}+\frac{2}{3}\max\{\frac{1-\frac{1}{3}-(2-\frac{i}{n}-\frac{j}{n})}{\frac{2}{3}},0\}\\
&=&\max\{\frac{i}{n}+\frac{j}{n}-1,\frac{1}{3}\}.
\end{eqnarray*}
Thus $T(\frac{i}{n},\frac{j}{n})\in{[\frac{1}{3},1]}$.
Now we take $\frac{q}{n}=\max\{\frac{i}{n}+\frac{j}{n}-1,\frac{1}{3}\}$. Then
$$T_{\alpha_2}(\varphi_{\alpha_2}(\frac{i}{n}),\varphi_{\alpha_2}(\frac{j}{n}))=\varphi_{\alpha_2}(\frac{q}{n}) \ \mbox{for}\ \mbox{all}\ \frac{i}{n},\frac{j}{n}\in[\frac{1}{3},1].$$}

\emph{Obviously, both $T_{\alpha_{1}}$ and $T_{\alpha_{2}}$ are continuous Archimedean t-norms. Therefore, by Theorem \ref{the 3.2}, $T^{[n]}$ is a t-norm.}

\emph{It is worth pointing out that $T$ is neither Archimedean nor strictly monotone, and it does not satisfy $T(x,y)<T_M(x,y)$ for all $x,y\in(0,1)$.}
\end{example}

These theorems and corollaries partly answer the question: What conditions guarantee that $T^{[n]}$ given by Eq.\eqref{eq 04} is a t-norm? Meanwhile, Corollary \ref{cor 3.2} theoretically demonstrates the issue in \cite[p.100]{EP2000}. Note that Example \ref{exp 2.1} shows that for t-norm $T=T_{P}$, $T^{[n]}$ is associative for $n=1,2,3$ and never associative for $n\geq{4}$. Motivated by Example \ref{exp 2.1}, we consider under what conditions $T^{[n]}$ given by Eq.\eqref{eq 04} is never associative while $T$ maintains its original condition, i.e., $T$ is strictly monotone.
\begin{theorem}\label{the 3.3}
If $T:[0,1]^2\rightarrow[0,1]$ is a strictly monotone t-norm with a convex additive generator $t$ satisfying $t(\frac{n-2}{n})+t(\frac{2}{n})\leq{t(\frac{1}{n})}$, then $T^{[n]}$ given by Eq.\eqref{eq 04} is never associative for $n\geq{4}$.
\end{theorem}
\begin{proof}
Choose $x,y\in{[\frac{n-1}{n},1)}, z\in{[\frac{2}{n},\frac{3}{n})}$, where $n\geq{4}$. Then by Eq.\eqref{eq 03} we have $f_{n}(x)=f_{n}(y)=\frac{n-1}{n}$ and $f_{n}(z)=\frac{2}{n}$. Subsequently, $$T^{[n]}(x,y)=f_{n}(T(f_{n}(x),f_{n}(y)))=f_{n}(T(\frac{n-1}{n},\frac{n-1}{n})),$$
$$T^{[n]}(y,z)=f_{n}(T(f_{n}(y),f_{n}(z)))=f_{n}(T(\frac{n-1}{n},\frac{2}{n})).$$
Note that $\frac{n-1}{n}=\frac{1}{2}(\frac{n-2}{n}+1)$ and $t(1)=0$.
Since $t$ is convex, we have $$t(\frac{n-1}{n})\leq \frac{1}{2}(t(\frac{n-2}{n})+t(1))=\frac{1}{2}t(\frac{n-2}{n}),$$ i.e., $2t(\frac{n-1}{n})\leq t(\frac{n-2}{n})$. Thus from (3.16) of \cite{EP2000}, we have
\begin{equation}\label{eq 00007}T(\frac{n-1}{n},\frac{n-1}{n})=t^{(-1)}(2t(\frac{n-1}{n}))\geq t^{(-1)}t(\frac{n-2}{n})=\frac{n-2}{n}.\end{equation}

Because $T$ is strictly monotone, we get $T(\frac{n-1}{n},\frac{n-1}{n})<\frac{n-1}{n}$ and $T(\frac{n-1}{n},\frac{2}{n})<\frac{2}{n}$, which together with Eq.\eqref{eq 00007} imply that
$$ \frac{n-2}{n}\leq{T(\frac{n-1}{n},\frac{n-1}{n})}<{\frac{n-1}{n}}.$$
Then by Eq.\eqref{eq 03}, we get \begin{equation}\label{eq010000}T^{[n]}(x,y)=f_{n}(T(\frac{n-1}{n},\frac{n-1}{n}))=\frac{n-2}{n}\mbox{ and } T^{[n]}(y,z)=f_{n}(T(\frac{n-1}{n},\frac{2}{n}))\in\{0,\frac{1}{n}\}.\end{equation}
Thus, we have $$T^{[n]}(x,T^{[n]}(y,z))=f_{n}(T(f_{n}(x),f_{n}(T^{[n]}(y,z))))=f_{n}(T(\frac{n-1}{n},T^{[n]}(y,z)))\leq f_{n}(T(\frac{n-1}{n},\frac{1}{n}))=0.$$
On the other hand, since $t(\frac{n-2}{n})+t(\frac{2}{n})\leq{t(\frac{1}{n})}$, we have $\frac{1}{n}\leq{T(\frac{n-2}{n},\frac{2}{n})}<\frac{2}{n}$, which together with Eq.\eqref{eq010000} implies that $$T^{[n]}(T^{[n]}(x,y),z)=f_{n}(T(f_{n}(T^{[n]}(x,y)),f_{n}(z)))=f_{n}(T(T^{[n]}(x,y),\frac{2}{n}))=f_{n}(T(\frac{n-2}{n},\frac{2}{n}))=\frac{1}{n}.$$

Therefore, $T^{[n]}(T^{[n]}(x,y),z)\neq{T^{[n]}(x,T^{[n]}(y,z))}$.
\end{proof}
\begin{remark}\label{rem 3.2}
\emph{\renewcommand{\labelenumi}{(\roman{enumi})}
In Theorem \ref{the 3.3}, $n\in \mathbb{N}$ must be greater than or equal to $4$. Otherwise, one of the three conditions (i) , (ii) and (iii) in Corollary \ref{cor 3.2} is true. Therefore, both $T^{[2]}$ and $T^{[3]}$ are associative.}
\end{remark}

\begin{proposition}\label{pro 3.3}\emph{(\cite{MR1978})}
Let $T:[0,1]^2\rightarrow[0,1]$ be a continuous Archimedean t-norm and $t:[0,1]\rightarrow{[0,\infty]}$ be its additive generator. Then the following are equivalent:\renewcommand{\labelenumi}{(\roman{enumi})}
\begin{enumerate}
\item $T$ satisfies the Lipschitz property with constant $1$.
\item $t$ is a convex function.
\end{enumerate}
\end{proposition}
\begin{corollary}\label{cor 3.3}
Let $T:[0,1]^2\rightarrow[0,1]$ be a strictly monotone t-norm with an additive generator $t:[0,1]\rightarrow{[0,\infty]}$ satisfying $t(\frac{n-2}{n})+t(\frac{2}{n})\leq{t(\frac{1}{n})}$. If $T$ satisfies the Lipschitz property with constant $1$, then $T^{[n]}$ given by Eq.\eqref{eq 04} is never associative for $n\geq{4}$.
\end{corollary}
\begin{proof}
It immediately follows from Proposition \ref{pro 3.3} and Theorem \ref{the 3.3}.
\end{proof}

\begin{corollary}\label{cor 3.4}
If $T:[0,1]^2\rightarrow[0,1]$ is a strictly monotone t-norm with twice differentiable additive generator $t:[0,1]\rightarrow{[0,\infty]}$ satisfying $t^{\prime\prime}\geq0$ and $t(\frac{n-2}{n})+t(\frac{2}{n})\leq{t(\frac{1}{n})}$, then $T^{[n]}$ given by Eq.\eqref{eq 04} is never associative for $n\geq{4}$.
\end{corollary}
\begin{proof}
Since $T$ is a strict t-norm with twice differentiable additive generator $t$ satisfying $t^{\prime\prime}\geq0$, it is clear that $t$ is a convex function. By Theorem \ref{the 3.3}, $T^{[n]}$ given by Eq.\eqref{eq 04} is never associative for $n\geq{4}$.
\end{proof}
\begin{example}\label{exp 3.2}
\emph{In \cite{EP2000}, a strict t-norm with a convex additive generator $t$ satisfying $t(\frac{n-2}{n})+t(\frac{2}{n})\leq{t(\frac{1}{n})}$ is, for example,
the Schweizer-Sklar t-norm $T_{\lambda}^{SS}$ for $\lambda\in{(-\infty,0)}$, the Hamacher t-norm $T_{\lambda}^{H}$ for $\lambda\in{[0,1)}$, the Frank
t-norm $T_{\lambda}^{F}$ for $\lambda\in{(0,1]}$, the Aczel-Alsina t-norm $T_{\lambda}^{AA}$ for $\lambda\in{[1,\infty)}$, the Dombi t-norm $T_{\lambda}^{D}$ for $\lambda\in{[1,\infty)}$ and the Sugeno-Weber t-norm $T_{\lambda}^{SW}$ for $\lambda=\infty$. Their respective additive generators are as follows.
$$t_{\lambda}^{SS}(x)=\left\{\begin{array}{ll}  -logx,  &\text{}\lambda=0,\\
\frac{1-x^{\lambda}}{\lambda},  &\text{} \lambda\in{(-\infty, 0)}, \end{array}\right.\quad
t_{\lambda}^{H}(x)=\left\{\begin{array}{ll}  \frac{1-x}{x},  &\text{}\lambda=0, \\
log(\frac{\lambda+(1-\lambda)x}{x}),  &\text{} \lambda\in{(0, 1)}, \end{array}\right.$$
$$t_{\lambda}^{F}(x)=\left\{\begin{array}{ll}  -logx,  &\text{}\lambda=1, \\
log(\frac{\lambda-1}{\lambda^{x}-1}),  &\text{} \lambda\in{(0, 1)}, \end{array}\right. \quad
t_{\lambda}^{AA}(x)=(-logx)^{\lambda}\mbox{ with }\lambda\in{[1,\infty)},$$
$$t_{\lambda}^{D}(x)=(\frac{1-x}{x})^{\lambda}\mbox{ with }\lambda\in{[1,\infty)}\mbox{ and }
t_{\lambda}^{SW}(x)=-logx\mbox{ with }\lambda=\infty.$$ Therefore, by Theorem \ref{the 3.3}, $T^{[n]}$ given by Eq.\eqref{eq 04} is never associative for $n\geq{4}$.}
\end{example}

The above results can be dually transformed into t-conorms. Indeed, similar to the construction of $f_{n}$ with $n\in\mathbb N$, we consider the function $g_{n}:[0,1]\rightarrow{\Gamma_{(n)}}$ defined by $g_{n}(x)=\frac{1}{n}{\lceil nx \rceil}$ (where $\lceil{ nx \rceil}$ denotes the smallest integer greater than or equal to $nx$). Through a simple computation, we have
\begin{equation*}
g_{n}(x)=\frac{1}{n}{\lceil nx \rceil}=\left\{\begin{array}{rcl}   0,  &\text{}x=0,   \\
\frac{1}{n}, &\text{}x\in{(0,\frac{1}{n}]},\\\\
 \frac{2}{n}, &\text{}x\in{(\frac{1}{n},\frac{2}{n}]},\\
 \vdots &\text{}\vdots\\
 \frac{n-1}{n}, &\text{}x\in{(\frac{n-2}{n},\frac{1}{n-1}]},\\\\
 1, &\text{}x\in{(\frac{n-1}{n},1]}.\\
  \end{array}\right.
 \end{equation*}
 It is easy to check that $g_{n}(x)$ is a non-decreasing, left-continuous function. Moreover, we immediately get \begin{equation*}f_{n}(x)=1-g_{n}(1-x).\end{equation*}
Let $S:[0,1]^{2}\rightarrow{[0,1]}$ be a t-conorm. Define a binary operation $S^{[n]}:[0,1]^{2}\rightarrow{[0,1]}$ by
\begin{equation}\label{eq 06}
S^{[n]}(x,y)=\left\{\begin{array}{ll}   g_{n}(S(g_{n}(x),g_{n}(y))),  &\text{}(x,y)\in(0,1]^{2},   \\
\max\{x,y\},  &\text{otherwise.}  \end{array}\right.
\end{equation}
Then one easily check that if $T:[0,1]^{2}\rightarrow{[0,1]}$ is the dual t-norm of $S$ then $T^{[n]}(x,y)=1-S^{[n]}(1-x,1-y)$ for any $x,y\in{[0,1]}$. Therefore, we dually have the following statements.
\begin{lemma}\label{lem 3.1}For any $n\in\mathbb N$, the binary operation $S^{[n]}:[0,1]^{2}\rightarrow{[0,1]}$ satisfies the axioms (S1), (S3) and (S4), but not necessarily (S2).
\end{lemma}

\begin{proposition}\label{pro 3.5}
Let $S:[0,1]^2\rightarrow[0,1]$ be a t-conorm and $S(\frac{i}{n},\frac{j}{n})\in{(\frac{n-2}{n},1]}$ for all $i,j\in{\{1,2,\dots, n-1, n\}}$. If one of the following condition holds,\renewcommand{\labelenumi}{(\roman{enumi})}
\begin{enumerate}
\item $S(x,y)>S_M(x,y)$ for all $x,y\in(0,1)$,
\item $S$ is Archimedean,
\item $S$ is strictly monotone,
\end{enumerate}
then $S^{[n]}$ given by Eq.\eqref{eq 06} is a t-conorm.
\end{proposition}

\begin{corollary}\label{cor 3.5}
Let $S:[0,1]^{2}\rightarrow{[0,1]}$ be a t-conorm. If one of the following condition holds,\renewcommand{\labelenumi}{(\roman{enumi})}
\begin{enumerate}
\item $S(x,y)>S_M(x,y)$ for all $x,y\in(0,1)$,
\item $S$ is Archimedean,
\item $S$ is strictly monotone,
\end{enumerate}
then both $S^{[2]}$ and $S^{[3]}$ are t-conorms.
\end{corollary}

\begin{proposition}\label{pro 3.6}
Let $S:[0,1]^2\rightarrow[0,1]$ be uniquely representable as an ordinal sum of continuous Archimedean t-conorms $(S_\alpha)_{\alpha\in A}$ with $A\neq\emptyset$. If for all $(\frac{i}{n},\frac{j}{n})\in[a_{\alpha},e_{\alpha}]^2$ and $\alpha\in A$, there exists $\frac{p}{n}\in[a_{\alpha},e_{\alpha}]$ such that $S_{\alpha}(\varphi_{\alpha}(\frac{i}{n}),\varphi_{\alpha}(\frac{j}{n}))=\varphi_{\alpha}(\frac{p}{n})$ where $\varphi_{\alpha}(x)=\frac{x-a_{\alpha}}{e_{\alpha}-a_{\alpha}}$ and $i, j, p\in{\{1,2,\dots, n-1, n\}}$, then $S^{[n]}$ given by Eq.\eqref{eq 06} is a t-conorm.
\end{proposition}
\begin{proposition}\label{pro 3.7}
If $S:[0,1]^2\to[0,1]$ is a strictly monotone t-conorm with a convex additive generator $s$ satisfying $s(\frac{n-2}{n})+s(\frac{2}{n})\leq{s(\frac{n-1}{n})}$,
then $S^{[n]}$ given by Eq.\eqref{eq 06} is never associative for $n\geq{4}$.
\end{proposition}
\begin{proof}
We take $T$ as the strictly monotone t-norm dual to the t-conorm $S$. Since $s$ is convex, it is easy to check that the additive generator $t$ of $T$ is convex. Since $t$ is dual to $s$, we have $t(\frac{n-2}{n})+t(\frac{2}{n})\leq{t(\frac{1}{n})}$.  Then by Theorem \ref{the 3.3}, $T^{[n]}$ is never associative for $n\geq{4}$. Since $S^{[n]}$ is dual to $T^{[n]}$, $S^{[n]}$ given by Eq.\eqref{eq 06} is never associative for $n\geq{4}$.
\end{proof}

\section{Conclusions}

In this article, we first identified a gap in a previous result concerning strictly monotone t-norms and showed that strict monotonicity alone does not imply the non-associativity of $T^{[n]}$. In particular, the product t-norm provides a counterexample for $n=2, 3$. We then established some sufficient conditions for $T^{[n]}$ to be a t-norm. In particular, for a broad class of t-norms, including Archimedean and strictly monotone t-norms, the discretizations $T^{[2]}$ and $T^{[3]}$ are associative. We also provided a condition for $T^{[n]}$ to be never associative when $T$ is strictly monotone. These results clarify how discretization affects the algebraic structure of t-norms and t-conorms, and in particular reveal a fundamental difference between the cases $n=2,3$ and $n\geq4$. Regrettably, this article just partly answers the two problem presented in Remark \ref{rem 2.3}.


\end{document}